\documentclass{article}
\usepackage[affil-it]{authblk}
\usepackage[utf8]{inputenc}
\usepackage{amsmath, amsfonts, amsthm, amssymb, latexsym}
\usepackage{geometry}
\usepackage{enumerate}
\usepackage{multirow}
\usepackage{rotating}
\usepackage{graphicx,color}
\usepackage{float}
\usepackage{mathrsfs}

\newtheorem{thm}{Theorem}[section]
\newtheorem{lemma}[thm]{Lemma}

\newtheorem{define}[thm]{Definition}
\newtheorem{prop}[thm]{Proposition}

\newtheorem*{remark}{Remark}

\title{On an inverse problem for a fractional semilinear elliptic equation involving a magnetic potential}
\author{Li Li}
\affil{Department of Mathematics, University of Washington,\\
Seattle, WA 98195, USA}
\date{}

\begin{document}

\maketitle

\noindent \textbf{ABSTRACT.}\, We study a class of fractional semilinear elliptic equations and formulate the corresponding Calder\'on problem. We determine the nonlinearity from the exterior partial measurements of the Dirichlet-to-Neumann map by using first order linearization and the Runge approximation property.

\section{Introduction}

The study of the fractional Calder\'on problem was initiated in \cite{ghosh2020calderon} where the authors considered an inverse problem for the fractional linear operator 
$$(-\Delta)^s+ q\qquad (0< s< 1).$$
See \cite{bhattacharyya2018inverse, cao2017simultaneously, cao2018determining, covi2020inverse, ghosh2017calderon, ghosh2020uniqueness, ruland2020fractional} for further studies based on \cite{ghosh2020calderon}.

Recently a fractional semilinear Calder\'on problem has been studied in 
\cite{lai2020inverse}. This inverse problem can be viewed as a nonlocal analogue of the classical semilinear Calder\'on problem studied in \cite{isakov1994global}. In \cite{lai2020inverse}, the authors considered the exterior Dirichlet problem
$$(-\Delta)^s u+ a(x, u)= 0\,\,\, \text{in}\,\,\Omega,\qquad
u|_{\Omega_e}= g$$
where $\Omega$ is a bounded domain with $C^{1,1}$ boundary and $\Omega_e:= \mathbb{R}^n\setminus\bar{\Omega}$, $n\geq 2$. Under some regularity assumptions on 
$a(\cdot, \cdot)$, the authors proved that the nonlinearity $a(\cdot, \cdot)$ can be
uniquely determined from the exterior partial measurements
of the Dirichlet-to-Neumann map 
$$\Lambda_a: g\to (-\Delta)^su_g|_{\Omega_e}.$$

In this paper, we extend the earlier result in \cite{lai2020inverse}. We study the generalized operator 
$\mathcal{R}^s_A$, which is formally defined by
\begin{equation}\label{Refml}
\mathcal{R}^s_A u(x):= 2\lim_{\epsilon\to 0^+}
\int_{\mathbb{R}^n\setminus B_\epsilon(x)}(u(x)- R_A(x, y)u(y))K(x,y)\,dy
\end{equation}
where $K(x, y)= c_{n, s}/|x-y|^{n+2s}$, 
$A$ is a fixed real vector-valued magnetic potential and
\begin{equation}\label{cos}
R_A(x, y):= \cos((x-y)\cdot A(\frac{x+y}{2})).
\end{equation}
Clearly $\mathcal{R}^s_A$ coincides with $(-\Delta)^s$ when $A= 0$
and $\mathcal{R}^s_A$ is the real part of the fractional magnetic Laplacian 
$(-\Delta)^s_A$ formally defined by 
$$(-\Delta)^s_A u(x):= 2\lim_{\epsilon\to 0^+}
\int_{\mathbb{R}^n\setminus B_\epsilon(x)}
(u(x)- e^{i(x-y)\cdot A(\frac{x+y}{2})}u(y))K(x,y)\,dy$$
(see for instance, \cite{d2018ground, squassina2016bourgain}) when $u$ is real-valued. We study the exterior Dirichlet problem
\begin{equation}\label{nonlineareDp}
\mathcal{R}^s_A u+ a(x, u)= 0\,\,\, \text{in}\,\,\Omega,\qquad
u|_{\Omega_e}= g.
\end{equation}
Our goal is to determine the nonlinearity $a(\cdot, \cdot)$ from the knowledge of the associated Dirichlet-to-Neumann map. Our inverse problem can be viewed as a semilinear analogue of the fractional
Calder\'on problem studied in \cite{li2020calderon}, which is a nonlocal analogue of the classical
Calder\'on problem for the magnetic Laplacian studied in 
\cite{ferreira2007determining, krupchyk2014uniqueness, nakamura1995global, sun1993inverse}.

Linearization is a standard technique used in solving the nonlinear Calder\'on problem. See for instance,
\cite{isakov1995global, sun2010inverse, sun1997inverse}.
In \cite{lai2020inverse}, the authors applied high order linearization
to prove their uniqueness theorem (see Section 3). See \cite{krupchyk2019partial, krupchyk2020remark, lassas2019inverse} for similar techniques used in solving the  
semilinear Calder\'on problem for local operators. 
In this paper,
we use the first order linearization in the Sobolev space $H^s(\mathbb{R}^n)$ and the 
Runge approximation property obtained in \cite{li2020calderon} instead to prove our uniqueness theorem.

To ensure that the exterior Dirichlet problem (\ref{nonlineareDp})
is well-posed for small $g$, we assume that $A$ satisfies some boundedness condition and 
the nonlinearity $a(x, z): \Omega\times \mathbb{R}\to \mathbb{R}$ satisfies\\
(i)\, $z\to a(\cdot, z)$ is analytic with values in $C^s(\Omega)$;\\
(ii)\, $a(x, 0)= 0$ and $\partial_z a(x, 0)\geq c> 0$
for some constant $c> 0$\\
so we have the Taylor's expansion
\begin{equation}\label{TaylorHolder}
a(x, z)= \sum^\infty_{k=1}a_k(x)\frac{z^k}{k!},\quad 
a_k(x)= \partial^k_z a(x, 0)\in C^s(\Omega)
\end{equation}
where the series converges in the H\"older space $C^s(\Omega)$ topology. 

Under the assumptions above, we can define the bounded solution operator
$Q_{A, a}: g\to u_g$
and the Dirichlet-to-Neumann map $\Lambda_{A, a}$ formally given by
\begin{equation}\label{DN}
\Lambda_{A, a}g:= \mathcal{R}^s_A (Q_{A, a}(g))|_{\Omega_e}
\end{equation}
to formulate corresponding Calder\'on problem. 

The following theorem is the main result in this paper.
\begin{thm}
Let $\Omega\subset \mathbb{R}^n$, $n\geq 2$ be a bounded domain with $C^{1,1}$ boundary.
Suppose $\Omega\cup \mathrm{supp}\,A\subset B_r(0)$ for some constant $r> 0$
and $||A||_{L^\infty(\mathbb{R}^n)}\leq \pi/(8\sqrt{n}r)$, $a^{(j)}$ satisfy (i) and (ii), $W_j$ are open sets s.t. $W_j\cap B_{3r}(0)= \emptyset$ ($j= 1, 2$). If 
$$\Lambda_{A, a^{(1)}}g|_{W_2}= \Lambda_{A, a^{(2)}}g|_{W_2},\qquad g\in C^\infty_c(W_1)$$
whenever $||g||_{C^2(\mathbb{R}^n)}$ is sufficiently small, then $a^{(1)}= a^{(2)}$ in $\Omega\times \mathbb{R}$.
\end{thm}

\begin{remark}
The nonlinear problem is reduced to the linear one when $a^{(j)}(x, z)= a^{(j)}_1(x)z$ ($j= 1, 2$). If this is the case, then the statement still holds after we replace $||A||_{L^\infty(\mathbb{R}^n)}\leq \pi/(8\sqrt{n}r)$ by the weakened assumption 
$A\in L^\infty(\mathbb{R}^n)$. See Theorem 1.1 in \cite{li2020calderon}.
\end{remark}

The rest of this paper is organized in the following way. In Section 2, we summarize the background knowledge. We prove that the nonlinear problem (\ref{nonlineareDp}) is well-posed in Section 4, based on the $L^\infty$ estimate and the H\"older regularity theorem for the corresponding linear problem proved in Section 3.
In Section 5, we prove the main theorem.~\\

\noindent \textbf{Acknowledgment.} The author is partly supported by National Science Foundation and would like to thank Professor Gunther Uhlmann for helpful discussions.

\section{Preliminaries}

Throughout this paper
\begin{itemize}
\item $n\geq 2$ denotes the space dimension and
$0< s< 1$ denotes the fractional power

\item $\Omega$ denotes a bounded domain with $C^{1,1}$ boundary and
$\Omega_e:= \mathbb{R}^n\setminus\bar{\Omega}$

\item $B_r(0)$ denotes the open ball centered at the origin with radius $r> 0$
and $\overline{B_{r}}(0)$ denotes the closure of $B_r(0)$

\item $A: \mathbb{R}^n\to \mathbb{R}^n$ denotes a real vector-valued magnetic potential

\item $c, C, C', C_1,\cdots$ denote positive constants (which may depend on some parameters but always independent of small constants $\epsilon, \rho$)

\item $\int\cdots\int= \int_{\mathbb{R}^n}\cdots\int_{\mathbb{R}^n}$

\item $X^*$ denotes the continuous dual space of $X$ and write
$\langle f, u\rangle= f(u)$ for $u\in X,\,f\in X^*$

\item $||\cdot||_{C^2(\mathbb{R}^n)}$ is defined by
$$||f||_{C^2(\mathbb{R}^n)}= 
\sum_{|\alpha|\leq 2}||\partial^\alpha f||_{L^\infty(\mathbb{R}^n)}.$$
\end{itemize}

\subsection{Function Spaces}

Throughout this paper we refer all function spaces to real-valued function spaces.

For $t\in \mathbb{R}$, we have Sobolev spaces
$$H^t(\mathbb{R}^n):= \{u\in \mathcal{S}'(\mathbb{R}^n):
\int (1+\vert \xi\vert^2)^t\vert \mathcal{F}u(\xi)\vert^2d\xi<\infty\}$$
where $\mathcal{F}$ is the Fourier transform and $\mathcal{S}'(\mathbb{R}^n)$ is the space of temperate distributions. We have the natural identification
$H^{-t}(\mathbb{R}^n)= H^t(\mathbb{R}^n)^*$.
Let $U$ be an open set and $F$ be a closed set in $\mathbb{R}^n$, 
$$H^t(U):= \{u|_U: u\in H^t(\mathbb{R}^n)\},\qquad 
H^t_F(\mathbb{R}^n):= \{u\in H^t(\mathbb{R}^n): \mathrm{supp}\,u\subset F\},$$
$$\tilde{H}^t(U):= 
\mathrm{the\,\,closure\,\,of}\,\, C^\infty_c(U)\,\,\mathrm{in}\,\, H^t(\mathbb{R}^n).$$
Since $\Omega$ is a bounded domain with $C^{1,1}$ boundary implies $\Omega$ is Lipschitz bounded, then 
$$\tilde{H}^t(\Omega)= H^t_{\bar{\Omega}}(\mathbb{R}^n).$$

For $0< s< 1$,
one of the equivalent forms of the norm $||\cdot||_{H^s(\mathbb{R}^n)}$ is
$$||u||_{H^s(\mathbb{R}^n)}:= (||u||^2_{L^2(\mathbb{R}^n)}+ \iint\frac{|u(x)-u(y)|^2}{|x-y|^{n+2s}}\,dxdy)^{1/2}.$$
We have the H\"older space $C^s(U):= C^{0, s}(U)$ equipped with the standard norm
given by
$$||f||_{C^s(U)}:= ||f||_{L^\infty(U)}+ 
\sup_{x\neq y, x,y\in U}\frac{|f(x)- f(y)|}{|x-y|^s}.$$

\subsection{The Operator $\mathcal{R}^s_A$}
In Section 1 we gave the formal pointwise definition of $\mathcal{R}^s_A$ in (\ref{Refml}). Now we do a formal computation to motivate the bilinear form definition of $\mathcal{R}^s_A$.

It is clear from (\ref{cos}) that $R_A(x,y)= R_A(y,x)$. 
Hence for real-valued $u, v$, we can formally compute that
$$2\iint_{\{|x-y|\geq \epsilon\}}(u(x)-R_A(x,y)u(y))v(x)K(x,y)\,dydx$$
$$= \iint_{\{|x-y|\geq \epsilon\}}[(u(x)-R_A(x,y)u(y))v(x)K(x,y)+
(u(y)-R_A(x,y)u(x))v(y)K(x,y)]\,dydx$$
$$= \mathrm{Re}\iint_{\{|x-y|\geq \epsilon\}}
(u(x)-e^{i(x-y)\cdot A(\frac{x+y}{2})}u(y))
(v(x)-e^{-i(x-y)\cdot A(\frac{x+y}{2})}v(y))K(x, y)\,dxdy.$$
Now let $\epsilon\to 0^+$.

\begin{define}
For real-valued $u, v$, we define $\mathcal{R}^s_A$ by the bilinear form

$$\langle \mathcal{R}^s_Au, v\rangle:= \mathrm{Re}\iint
(u(x)-e^{i(x-y)\cdot A(\frac{x+y}{2})}u(y))
(v(x)-e^{-i(x-y)\cdot A(\frac{x+y}{2})}v(y))K(x, y)\,dxdy$$
\begin{equation}\label{blrsa}
= 2\iint(u(x)-R_A(x,y)u(y))v(x)K(x,y)\,dxdy.
\end{equation}
\end{define}
It is easy to verify that
$$\langle \mathcal{R}^s_Au, v\rangle= \langle \mathcal{R}^s_Av, u\rangle.$$

\begin{define}
We define the magnetic Sobolev norm $||\cdot||_{H^s_A}$ by
$$||u||_{H^s_A}:= (||u||^2_{L^2}+ [u]^2_{H^s_A})^{1/2}$$
where $[u]_{H^s_A}:= \langle \mathcal{R}^s_Au, u\rangle^\frac{1}{2}$.
\end{define}
This norm was introduced in \cite{d2018ground, squassina2016bourgain}. As we mentioned in Section 1, $\mathcal{R}^s_A$ is the real part of the fractional magnetic Laplacian, whose properties have been studied in \cite{li2020calderon}. In fact, Lemma 3.3 and Proposition 3.4 in \cite{li2020calderon} imply the following proposition.
\begin{prop}
Suppose $0< s< 1$ and $A\in L^\infty(\mathbb{R}^n)$, then we have the norm equivalence
$||\cdot||_{H^s_A}\sim ||\cdot||_{H^s}$ and the operator
$$\mathcal{R}^s_A: H^s(\mathbb{R}^n)\to H^{-s}(\mathbb{R}^n)$$
is linear and bounded.
\end{prop}

From now on we always assume $A\in L^\infty(\mathbb{R}^n)$. Corollary 5.3 in 
\cite{li2020calderon} implies the following proposition, which will be used later in the proof of the main theorem.
\begin{prop}
Suppose $\Omega\cup \mathrm{supp}\,A\subset B_r(0)$ for some $r> 0$, $W$ is an open set s.t. $W\setminus \overline{B_{3r}}(0)\neq \emptyset$. If
$$u\in \tilde{H}^s(\Omega),\qquad \mathcal{R}^s_{A}u|_W= 0$$
then $u= 0$ in $\mathbb{R}^n$.
\end{prop}

\section{The Linear Exterior Problem}
Throughout this section we assume $0< c\leq q(x)\in L^\infty(\Omega)$.

We first recall some results in \cite{li2020calderon}. Based on Proposition 4.3 and Proposition 4.5 in \cite{li2020calderon}, we have the following proposition.

\begin{prop}
The bilinear form
$$B_{A, q}(u, v):= \langle \mathcal{R}^s_Au, v\rangle+ \int_\Omega quv$$
is coercive and bounded on $\tilde{H}^s(\Omega)\times \tilde{H}^s(\Omega)$. 
The exterior problem
\begin{equation}\label{Rsalep}
\left\{
\begin{aligned}
(\mathcal{R}^s_A+ q)u&= 0\quad \text{in}\,\,\Omega\\
u&= g\quad \text{in}\,\,\Omega_e\\
\end{aligned}
\right.
\end{equation}
has a unique (weak) solution $u_g\in H^s(\mathbb{R}^n)$ for each $g\in H^s(\mathbb{R}^n)$ and the solution operator 
$$P_{A, q}: g\to u_g$$ is bounded on $H^s(\mathbb{R}^n)$.
\end{prop}

Proposition 5.4 in \cite{li2020calderon} implies the Runge approximation property of $\mathcal{R}^s_A+ q$, which will be used later in the proof of the main theorem.
\begin{prop}
Suppose $\Omega\cup \mathrm{supp}\,A\subset B_r(0)$ for some $r> 0$, $W$ is an open set s.t. $W\subset \Omega_e$ and $W\setminus \overline{B_{3r}}(0)\neq \emptyset$, then
$$S:= \{P_{A, q}f|_\Omega: f\in C^\infty_c(W)\}$$
is dense in $L^2(\Omega)$.
\end{prop}

Next we prove an $L^\infty$ estimate and a H\"older regularity theorem, which will be useful in later sections when we deal with the nonlinear problem.

\subsection{$L^\infty$ Estimate}
\begin{lemma}
If $g\in C^\infty_c(\mathbb{R}^n)$, then $(-\Delta)^s g\in L^\infty(\mathbb{R}^n)$ and
$$||(-\Delta)^s g||_{L^\infty(\mathbb{R}^n)}\leq C||g||_{C^2(\mathbb{R}^n)}$$
\end{lemma}
\begin{proof}
For $g\in C^\infty_c(\mathbb{R}^n)$, we have
$$(-\Delta)^s g(x)= c_{n, s}\int\frac{2g(x)- g(x+y)- g(x-y)}{|y|^{n+2s}}dy$$
(see for instance, Lemma 3.2 in \cite{di2011hitchhiker}) 
so by using Taylor's expansion, we have
$$|(-\Delta)^s g(x)|\leq c_{n, s}(\int_{|y|\leq 1}+ \int_{|y|> 1})
\frac{|2g(x)- g(x+y)- g(x-y)|}{|y|^{n+2s}} dy\leq C||g||_{C^2(\mathbb{R}^n)}.$$
\end{proof}

\begin{lemma}
If $A, g\in L^\infty(\mathbb{R}^n)$, then $((-\Delta)^s- \mathcal{R}^s_A) g\in L^\infty(\mathbb{R}^n)$ and
$$||((-\Delta)^s- \mathcal{R}^s_A) g||_{L^\infty(\mathbb{R}^n)}\leq C||g||_{L^\infty(\mathbb{R}^n)}.$$
\end{lemma}
\begin{proof}
Note that 
$$0\leq 1- R_A(x,y)= 2\sin^2(\frac{1}{2}(x-y)\cdot A(\frac{x+y}{2}))\leq 
C_A\min\{1, |x-y|^2\}$$
so we have 
$$|((-\Delta)^s- \mathcal{R}^s_A) g(x)|\leq
\int (1- R_A(x, y))K(x, y)|g(y)|dy$$
$$= (\int_{|y-x|\leq 1}+ \int_{|y-x|> 1})(1- R_A(x, y))K(x, y)|g(y)|dy
\leq C||g||_{L^\infty(\mathbb{R}^n)}.$$
\end{proof}

From now on we always assume 
$\Omega\cup \mathrm{supp}\,A\subset B_r(0)$ for some constant $r> 0$
and $||A||_{L^\infty(\mathbb{R}^n)}\leq \pi/(8\sqrt{n}r)$. This coincides with the assumption on $A$ in the statement of Theorem 1.1.

Note that under this assumption, we have
$$0\leq R_A(x, y)\leq 1,\qquad 
(x, y)\in (\mathbb{R}^n\times \mathbb{R}^n)\setminus (\Omega_e\times \Omega_e).$$
In fact, if $(x,y)\in \Omega\times \mathbb{R}^n$, then\\
(i)\, $y\in B_{3r}(0)$, which implies $|x-y|\leq 4r$, 
$|(x-y)\cdot A(\frac{x+ y}{2})|\leq \frac{\pi}{2}$;\\
(ii)\, $y\notin B_{3r}(0)$, which implies $|\frac{x+y}{2}|\geq r$, $R_A= 1$.\\
By symmetry of $R_A$, we know the claim also holds for $(x,y)\in 
\mathbb{R}^n\times \Omega$.

The following two propositions generalize Proposition 3.1 and Proposition 3.3 in
\cite{lai2019global}.

\begin{prop}
Suppose $0< c\leq q(x)\in L^\infty(\Omega)$. If $u\in H^s(\mathbb{R}^n)$ solves the exterior problem 
\begin{equation}
\left\{
\begin{aligned}
(\mathcal{R}^s_A+ q)u&= f\quad \text{in}\,\,\Omega\\
u&= g\quad \text{in}\,\,\Omega_e\\
\end{aligned}
\right.
\end{equation}
for $0\leq f\in H^{-s}(\Omega)$ and $0\leq g|_{\Omega_e}\in L^\infty(\Omega_e)$, then $u\geq 0$.
\end{prop}
\begin{proof}
Write $u= u^+- u^-$ where $u^+= \max\{u, 0\}$ and $u^-= \max\{-u, 0\}$. Note that
$$|u^+(x)- u^+(y)| +|u^-(x)- u^-(y)|= |u(x)- u(y)|$$
so $u^+, u^-\in H^s(\mathbb{R}^n)$ and $g|_{\Omega_e}\geq 0$ implies $u^-\in \tilde{H}^s(\Omega)$, so we have
$$\langle\mathcal{R}^s_A u, u^-\rangle+
\int_\Omega q u u^-= f(u^-).$$
Now write
$$\langle\mathcal{R}^s_A u, u^-\rangle 
= 2\iint(u(x)-R_A(x,y)u(y))u^-(x)K(x,y)\,dxdy$$
$$= 2(\int_\Omega\int_\Omega+ \int_{\Omega_e}\int_\Omega) (u(x)-R_A(x,y)u(y))u^-(x)K(x,y)\,dxdy =: I_1+ I_2.$$

Since $u^-u\leq 0$, then we have
$$I_2= 2\int_{\Omega_e}\int_\Omega(u(x)-R_A(x,y)g(y))u^-(x)K(x,y)\,dxdy\leq 0,\quad
\int_\Omega q u u^-\leq 0.$$

Note that
$$I_1= 2\int_\Omega\int_\Omega(u(x)-R_A(x,y)u(y))u^-(x)K(x,y)\,dxdy$$
$$= 2\int_\Omega\int_\Omega[(u^+(x)-R_A(x,y)u^+(y))u^-(x)
-(u^-(x)-R_A(x,y)u^-(y))u^-(x)]K(x,y)\,dxdy$$
$$= -2\int_\Omega\int_\Omega [R_A(x,y)u^+(y)u^-(x)+(u^-(x)-R_A(x,y)u^-(y))u^-(x)]K(x,y)\,dxdy$$
$$\leq -2\int_\Omega\int_\Omega(u^-(x)-R_A(x,y)u^-(y))u^-(x)K(x,y)\,dxdy$$
$$= -\int_\Omega\int_\Omega[(u^-(x)-R_A(x,y)u^-(y))u^-(x)
+(u^-(y)-R_A(x,y)u^-(x))u^-(y)]K(x,y)\,dxdy$$
$$= -\int_\Omega\int_\Omega(|u^-(x)|^2+ |u^-(y)|^2- 2R_A(x,y)u^-(x)u^-(y))K(x,y)\,dxdy$$
$$\leq -\int_\Omega\int_\Omega|u^-(x)-u^-(y)|^2K(x,y)\,dxdy.$$
Since $f(u^-)\geq 0$, then the only possibility is
$$\int_\Omega\int_\Omega|u^-(x)-u^-(y)|^2K(x,y)\,dxdy= 0,$$
which implies $u^-$ is a non-negative constant $c_0$ in $\Omega$. Now we show $c_0$ has to be $0$. 

Otherwise $u= -c_0< 0$ in $\Omega$. In this case, for $x\in \Omega$, by pointwise definition we have
$$\mathcal{R}^s_A u(x)= 2\lim_{\epsilon\to 0^+}
(\int_{\Omega\setminus B_\epsilon(x)}+ \int_{\Omega_e})
(u(x)- R_A(x, y)u(y))K(x,y)\,dy$$
$$= 2\int_\Omega (1- R_A(x,y))(-c_0)K(x,y)dy+ 
2\int_{\Omega_e}(-c_0- R_A(x,y)g(y))K(x,y)dy\leq 0.$$
Both integrals converge since $g|_{\Omega_e}\in L^\infty(\Omega_e)$ and 
$0\leq 1- R_A(x,y)\leq C_A|x-y|^2$.
Now we have got the contradiction $$f= \mathcal{R}^s_A u+ qu< 0\quad\mathrm{in}\,\,\Omega.$$
\end{proof}

\begin{prop}
Suppose $0< c\leq q(x)\in L^\infty(\Omega)$.
If $u\in H^s(\mathbb{R}^n)$ solves the exterior problem 
\begin{equation}
\left\{
\begin{aligned}
(\mathcal{R}^s_A+ q)u&= f\quad \text{in}\,\,\Omega\\
u&= g\quad \text{in}\,\,\Omega_e\\
\end{aligned}
\right.
\end{equation}
for $f\in L^\infty(\Omega)$ and $g\in C^\infty_c(\Omega_e)$, then 
$$||u||_{L^\infty}\leq C||f||_{L^\infty(\Omega)}+ ||g||_{L^\infty(\Omega_e)}.$$
\end{prop}
\begin{proof}
Fix a function $\phi\in C^\infty_c(\mathbb{R}^n)$ s.t. $0\leq \phi\leq 1$ and
$\phi= 1$ on $\bar{\Omega}\cup \mathrm{supp}\,g$.

It is clear from the pointwise definition of $\mathcal{R}^s_A$ that
$\mathcal{R}^s_A \phi\geq 0$ in $\Omega$ so
$$(\mathcal{R}^s_A+ q)\phi\geq c\quad\mathrm{in}\,\,\Omega.$$
Now let 
$\tilde{\phi}:= (\frac{1}{c}||f||_{L^\infty(\Omega)}+ ||g||_{L^\infty(\Omega_e)})\phi$, then
$\tilde{\phi}\pm u\geq 0$ in $\Omega_e$ and 
$$(\mathcal{R}^s_A+ q)(\tilde{\phi}\pm u)\geq 0\quad\mathrm{in}\,\,\Omega$$
so $|u|\leq \tilde{\phi}$ by the previous proposition.
\end{proof}

\subsection{H\"older Regularity}
\begin{prop}
(Proposition 1.1 in \cite{ros2014dirichlet}) If $u\in \tilde{H}^s(\Omega)$ and $(-\Delta)^su= f\in L^\infty(\Omega)$, then 
$u\in C^s(\mathbb{R}^n)$ and 
$$||u||_{C^s(\mathbb{R}^n)}\leq C||f||_{L^\infty(\Omega)}.$$
\end{prop}

Based on the proposition above, we now prove the H\"older regularity theorem
for the linear exterior problem (\ref{Rsalep}).

\begin{prop}
Suppose $\Omega\cup \mathrm{supp}\,A\subset B_r(0)$ for some $r> 0$
and $||A||_{L^\infty(\mathbb{R}^n)}\leq \pi/(8\sqrt{n}r)$, $0< c\leq q(x)\in L^\infty(\Omega)$ for some $c> 0$, $W\cap B_{3r}(0)= \emptyset$. If $g\in C^\infty_c(W)$, then 
$u= P_{A,q}g\in C^s(\mathbb{R}^n)$ where $P_{A,q}$ is solution operator associated with  (\ref{Rsalep}) and 
$$||u||_{C^s(\mathbb{R}^n)}\leq C||g||_{C^2(\mathbb{R}^n)}.$$
\end{prop}
\begin{proof}
Note that $v:= u- g\in \tilde{H}^s(\Omega)$ and
$$(-\Delta)^sv= ((-\Delta)^s- \mathcal{R}^s_A)v- \mathcal{R}^s_Ag- qu\quad\mathrm{in}\,\,\Omega.$$
By Proposition 3.6, $||v||_{L^\infty}\leq C_1||g||_{L^\infty}$. By Lemma 3.4,
$$||((-\Delta)^s- \mathcal{R}^s_A) v||_{L^\infty(\mathbb{R}^n)}\leq C_2||v||_{L^\infty}\leq C_3||g||_{L^\infty}.$$
Since $W\cap B_{3r}(0)= \emptyset$, then $|\frac{x+y}{2}|\geq r$, $R_A(x, y)= 1$
for $x\in \Omega, y\in W$ so
$$\mathcal{R}^s_Ag= (-\Delta)^sg\quad\mathrm{in}\,\,\Omega.$$
By Lemma 3.3, $||\mathcal{R}^s_A g||_{L^\infty(\Omega)}\leq C'||g||_{C^2(\mathbb{R}^n)}$.
Hence
$$||((-\Delta)^s- \mathcal{R}^s_A)v- \mathcal{R}^s_Ag- qu||_{L^\infty(\Omega)}
\leq C''||g||_{C^2(\mathbb{R}^n)}$$
Now by the proposition above, we have 
$||v||_{C^s(\mathbb{R}^n)}\leq C||g||_{C^2(\mathbb{R}^n)}$.
\end{proof}

\section{The Nonlinear Exterior Problem}
Now we consider the nonlinear exterior problem
\begin{equation}\label{Rsanonlinear}
\left\{
\begin{aligned}
\mathcal{R}^s_A u+ a(x, u)&= 0\quad \text{in}\,\,\Omega\\
u&= g\quad \text{in}\,\,\Omega_e.\\
\end{aligned}
\right.
\end{equation}

Recall that we assume the nonlinearity $a(\cdot, \cdot)$ has the Taylor's expansion
(\ref{TaylorHolder}) with coefficients $a_k(x)\in C^s(\Omega)$ and we also assume $a_1(x)\geq c> 0$. We write
$$R_m(x, z):= \sum^\infty_{k= m+1}\frac{a_k(x)}{k!}z^k.$$

We note that $C^s(\Omega)$ is an algebra since for $u_1, u_2\in C^s(\Omega)$, we have
$$||u_1u_2||_{C^s(\Omega)}\leq C_0
(||u_1||_{C^s(\Omega)}||u_2||_{L^\infty(\Omega)}+
||u_1||_{L^\infty(\Omega)}||u_2||_{C^s(\Omega)})$$
(see Theorem A.7 in \cite{hormander1976boundary}) so 
$$||u_1u_2||_{C^s(\Omega)}\leq 2C_0||u_1||_{C^s(\Omega)}||u_2||_{C^s(\Omega)}.$$

Also note that by Cauchy's estimate, we have
$$||a_k||_{C^s(\Omega)}\leq \frac{k!}{R^k}
\sup_{z\in\mathbb{C}, |z|=R}||a(\cdot, z)||_{C^s(\Omega)},\quad R> 0.$$

Based on the estimates above, we have the following estimates when we choose
$R= \max\{4C_0, 1\}$.
\begin{prop}
If $||u||_{C^s(\Omega)}\leq 1$, then 
$$\sum^\infty_{k= m+1}||\frac{a_k(x)}{k!}u^k||_{C^s(\Omega)}
\leq (\sum^\infty_{k= m+1}\frac{1}{2^k})\sup_{z\in\mathbb{C}, |z|=R}||a(\cdot, z)||_{C^s(\Omega)}||u||^{m+1}_{C^s(\Omega)},$$
$$\sum^\infty_{k= m+1}||\frac{a_k(x)}{(k-1)!}u^{k-1}||_{C^s(\Omega)}
\leq (\sum^\infty_{k= m+1}\frac{k}{2^{k-1}})\sup_{z\in\mathbb{C}, |z|=R}||a(\cdot, z)||_{C^s(\Omega)}||u||^m_{C^s(\Omega)}.$$
\end{prop}

The following proposition is an analogue of Theorem 2.1 in \cite{lai2020inverse}.
\begin{prop}
Suppose $\Omega\cup \mathrm{supp}\,A\subset B_r(0)$ for some $r> 0$
and $||A||_{L^\infty(\mathbb{R}^n)}\leq \pi/(8\sqrt{n}r)$, $W\cap B_{3r}(0)= \emptyset$ and $g\in C^\infty_c(W)$. There exists a small
constant $\rho >0$ s.t. if $||g||_{C^2(\mathbb{R}^n)}\leq \rho$, then the nonlinear 
exterior problem
(\ref{Rsanonlinear})
has a unique solution $u\in H^s(\mathbb{R}^n)\cap C^s(\mathbb{R}^n)$ satisfying
$$(u- P_{A, a_1}g)\in
M:=\{v\in C^s(\mathbb{R}^n): v|_{\Omega_e}= 0, ||v||_{C^s(\mathbb{R}^n)}\leq \rho\}.$$
Moreover, we have
$$||u||_{C^s(\mathbb{R}^n)}\leq C||g||_{C^2(\mathbb{R}^n)}.$$
\end{prop}
\begin{proof}
Let $u_0:= P_{A, a_1}g$, then by Proposition 3.8 we have
$$||u_0||_{C^s(\mathbb{R}^n)}\leq C_1||g||_{C^2(\mathbb{R}^n)},$$
and the nonlinear 
exterior problem (\ref{Rsanonlinear}) can be written as
\begin{equation}
\left\{
\begin{aligned}
\mathcal{R}^s_A (u- u_0)+ a_1(x)(u- u_0)&= -R_1(x, u)\quad \text{in}\,\,\Omega\\
u- u_0&= 0\quad \text{in}\,\,\Omega_e.\\
\end{aligned}
\right.
\end{equation}

Now for $f\in L^\infty(\Omega)$, we consider the solution operator
$J: f\to u_f\in \tilde{H}^s(\Omega)$ of
$$\mathcal{R}^s_A u+ a_1(x)u= f\quad \text{in}\,\,\Omega.$$
We write
$$(-\Delta)^s u= ((-\Delta)^s-\mathcal{R}^s_A)u- a_1(x)u+ f\quad \text{in}\,\,\Omega,$$
so by Lemma 3.4, Proposition 3.6 and 3.7, we have $J(f)\in C^s(\mathbb{R}^n)$ and
$$||J(f)||_{C^s(\mathbb{R}^n)}\leq C_2||f||_{L^\infty(\Omega)}.$$
We define maps $G, F$ by 
$$G(v):= R_1(x, u_0+ v),\qquad F:= J\circ G.$$

We will show that
$F$ is a contraction map on the complete metric space $M$ for 
small $\rho$, which will be chosen later.
In fact, for small $\rho$ and $v\in M$, we have
$$||F(v)||_{C^s(\mathbb{R}^n)}\leq C_2||G(v)||_{L^\infty(\Omega)}
= C_2||R_1(x, u_0+ v)||_{L^\infty(\Omega)}\leq C'_2||u_0+ v||^2_{C^s(\Omega)}
\leq C''_2\rho^2.$$
Here we use the first estimate in Proposition 4.1 and the constant $C''_2$
is independent of $\rho$.

Also for small $\rho$ and $v_1, v_2\in M$, we have
$$||F(v_1)- F(v_2)||_{C^s(\mathbb{R}^n)}\leq C_2||G(v_1)- G(v_2)||_{L^\infty(\Omega)}$$
$$= C_2||R_1(x, u_0+ v_1)- R_1(x, u_0+ v_2)||_{L^\infty(\Omega)}$$
$$\leq ||v_1- v_2||_{L^\infty(\Omega)}
\sum^\infty_{k= 2}||\frac{a_k(x)}{(k-1)!}
(|u_0+ v_1|^{k-1}+ |u_0+ v_1|^{k-1})||_{L^\infty(\Omega)}$$
$$\leq C_3||v_1- v_2||_{L^\infty(\Omega)}(||u_0+ v_1||_{C^s(\Omega)}+ 
||u_0+ v_2||_{C^s(\Omega)})$$
$$\leq C_4\rho ||v_1- v_2||_{L^\infty(\Omega)}.$$
Here we use the inequality 
$$|a^m- b^m|\leq m|a- b|(|a|^{m-1}+ |b|^{m-1})$$
and the second estimate in Proposition 4.1. The constant $C_4$ is independent of $\rho$.

Hence, for small $\rho< 1/(C''_2+ C_4)$, $F$ is a contraction map on $M$ so by Banach
fixed-point theorem, there exists a unique $v_0\in M$ s.t. $F(v_0)= v_0$.

Now note that $v_0= F(v_0)\in \tilde{H}^s(\Omega)$ and
$$||v_0||_{C^s(\mathbb{R}^n)}= ||F(v_0)||_{C^s(\mathbb{R}^n)}
\leq C'_2||u_0+ v_0||^2_{C^s(\Omega)}
\leq C'_3\rho(||u_0||_{C^s(\Omega)}+ ||v_0||_{C^s(\Omega)})$$
where the constant $C'_3$ is independent of $\rho$.
Hence, for small $\rho< 1/(2C'_3)$, we have
$$||v_0||_{C^s(\mathbb{R}^n)}\leq 2C'_3\rho||u_0||_{C^s(\mathbb{R}^n)}$$
and then $u:= u_0+ v_0$ satisfies
$$||u||_{C^s(\mathbb{R}^n)}\leq C||g||_{C^2(\mathbb{R}^n)}.$$
\end{proof}

\section{The Inverse Problem}
From now on, we denote the bounded solution operator associated with
(\ref{Rsanonlinear}) by $Q_{A, a}$. 

Proposition 4.2 ensures that the Dirichlet-to-Neumann map $\Lambda_{A, a}$ given by (\ref{DN}) is well-defined for $g$ satisfying the condition assumed in the statement of the proposition. We remark that $\Lambda_{A, a}$ coincides the Dirichlet-to-Neumann map
defined in Subsection 2.3 in \cite{lai2020inverse} when $A= 0$.

The first order linearization in $H^s(\mathbb{R}^n)$ will be useful when we prove Theorem 1.1
later.

\begin{prop}
Suppose $\Omega\cup \mathrm{supp}\,A\subset B_r(0)$ for some $r> 0$
and $||A||_{L^\infty(\mathbb{R}^n)}\leq \pi/(8\sqrt{n}r)$, $W\cap B_{3r}(0)= \emptyset$ and $g\in C^\infty_c(W)$, then $$Q_{A, a}(\epsilon g)/\epsilon\to P_{A, a_1}g$$ in $H^s(\mathbb{R}^n)$ as $\epsilon\to 0$.
\end{prop}
\begin{proof}
Write $u_{\epsilon, g}:= Q_{A, a}(\epsilon g)$ and $u_g:= P_{A, a_1}g$ for sufficiently small $\epsilon$. 

Note that
$v_{\epsilon, g}:= u_g- \frac{u_{\epsilon, g}}{\epsilon}\in \tilde{H}^s(\Omega)$
and we have
$$\mathcal{R}^s_A v_{\epsilon, g}+ a_1(x)v_{\epsilon, g}
= \frac{1}{\epsilon}R_1(x, u_{\epsilon, g})\quad \text{in}\,\,\Omega.$$
Now choose $v_{\epsilon, g}$ as a test function, then by Proposition 2.3 we have
$$\langle \mathcal{R}^s_A v_{\epsilon, g}+ a_1 v_{\epsilon, g}, 
v_{\epsilon, g}\rangle
\geq [v_{\epsilon, g}]^2_{H^s_A}+ c||v_{\epsilon, g}||^2_{L^2(\Omega)}
\geq c'||v_{\epsilon, g}||^2_{H^s}$$
and by Proposition 4.1 and 4.2, we have
$$|\langle\frac{1}{\epsilon}R_1(x, u_{\epsilon, g}), v_{\epsilon, g}\rangle|
\leq \frac{C}{\epsilon}||R_1(x, u_{\epsilon, g})||_{L^\infty(\Omega)}||v_{\epsilon, g}||_{L^2(\Omega)}$$
$$\leq \frac{C'}{\epsilon}||u_{\epsilon, g}||^2_{C^s(\mathbb{R}^n)}||v_{\epsilon, g}||_{L^2(\Omega)}
\leq C''\epsilon||g||^2_{C^2(\mathbb{R}^n)}||v_{\epsilon, g}||_{L^2(\Omega)}.$$
Hence we have
$$||v_{\epsilon, g}||_{H^s}\leq C'''\epsilon||g||^2_{C^2(\mathbb{R}^n)}.$$
Now it is clear that $v_{\epsilon, g}\to 0$ in $H^s(\mathbb{R}^n)$ as $\epsilon\to 0$.
\end{proof}

Now we are ready to prove Theorem 1.1.

\begin{proof}
Write $u^{(j)}_{\epsilon, g}:= Q_{A, a^{(j)}}(\epsilon g)$ 
and $u^{(j)}_g:= P_{A, a^{(j)}_1}(g)$ for $g\in C^\infty_c(W_1)$ 
and sufficiently small $\epsilon$.

By the assumption, we have
$$\mathcal{R}^s_A u^{(1)}_{\epsilon, g}= 
\mathcal{R}^s_A u^{(2)}_{\epsilon, g}\quad
\mathrm{in}\,\,W_2.$$
Since $u^{(1)}_{\epsilon, g}= u^{(2)}_{\epsilon, g}= \epsilon g$ in $\Omega_e$,
then by Proposition 2.4 we have
$u^{(1)}_{\epsilon, g}= u^{(2)}_{\epsilon, g}=: u_{\epsilon, g}$ in $\mathbb{R}^n$ so
$$\mathcal{R}^s_A u_{\epsilon, g}+ a^{(j)}(x, u_{\epsilon, g})= 0\quad
\mathrm{in}\,\,\Omega\quad (j= 1, 2),$$
which implies
$$(a^{(1)}_1(x)- a^{(2)}_1(x))u_{\epsilon, g}= 
R^{(2)}_1(x, u_{\epsilon, g})- R^{(1)}_1(x, u_{\epsilon, g})\quad
\mathrm{in}\,\,\Omega.$$

Now note that
$$||a^{(1)}_1(x)- a^{(2)}_1(x)||_{L^2(\Omega)}\leq 
||(a^{(1)}_1(x)- a^{(2)}_1(x))(1- \frac{u_{\epsilon, g}}{\epsilon})||_{L^2(\Omega)}
+ \frac{1}{\epsilon}||(a^{(1)}_1(x)- a^{(2)}_1(x))u_{\epsilon, g}||_{L^2(\Omega)}$$
\begin{equation}\label{Est}
\leq ||a^{(1)}_1(x)- a^{(2)}_1(x)||_{L^\infty(\Omega)}||1- \frac{u_{\epsilon, g}}{\epsilon}||_{L^2(\Omega)}
+ \frac{1}{\epsilon}||(a^{(1)}_1(x)- a^{(2)}_1(x))u_{\epsilon, g}||_{L^2(\Omega)}
\end{equation}

For given $\delta> 0$, by Proposition 3.2 we can choose 
$g\in C^\infty_c(W_1)$ s.t. 
$$||1- u_g||_{L^2(\Omega)}\leq \delta.$$ 
For this chosen $g$, we have 
$$||1- \frac{u_{\epsilon, g}}{\epsilon}||_{L^2(\Omega)}\leq 2\delta$$
for small $\epsilon$ by Proposition 5.1 and we also have
$$\frac{1}{\epsilon}||(a^{(1)}_1(x)- a^{(2)}_1(x))u_{\epsilon, g}||_{L^2(\Omega)}\leq
\frac{C}{\epsilon}||R^{(2)}_1(x, u_{\epsilon, g})- R^{(1)}_1(x, u_{\epsilon, g})||
_{L^\infty(\Omega)}$$
$$\leq \frac{C'}{\epsilon}||u_{\epsilon, g}||^2_{C^s(\Omega)}\leq
C''\epsilon||g||^2_{C^2(\mathbb{R}^n)}$$
for small $\epsilon$ by Proposition 4.1 and 4.2.

Now let $\epsilon\to 0$ in (\ref{Est}), then we have
$$||a^{(1)}_1(x)- a^{(2)}_1(x)||_{L^2(\Omega)}\leq 
2\delta ||a^{(1)}_1(x)- a^{(2)}_1(x)||_{L^\infty(\Omega)}.$$
Since $\delta$ is arbitrary, then $a^{(1)}_1= a^{(2)}_1=: a_1$.

Iteratively, once we have shown $a^{(1)}_j= a^{(2)}_j$ ($1\leq j\leq l-1$), then we have
$$\frac{1}{l!}(a^{(1)}_l(x)- a^{(2)}_l(x))u^l_{\epsilon, g}= 
R^{(2)}_l(x, u_{\epsilon, g})- R^{(1)}_l(x, u_{\epsilon, g})\quad
\mathrm{in}\,\,\Omega.$$

Now note that
$$|||a^{(1)}_l(x)- a^{(2)}_l(x)|^\frac{1}{l}||_{L^2(\Omega)}$$
$$\leq |||a^{(1)}_l(x)- a^{(2)}_l(x)|^\frac{1}{l}||_{L^\infty(\Omega)}||1- \frac{u_{\epsilon, g}}{\epsilon}||_{L^2(\Omega)}
+ \frac{1}{\epsilon}|||a^{(1)}_l(x)- a^{(2)}_l(x)|^\frac{1}{l}u_{\epsilon, g}||_{L^2(\Omega)}.$$

For given $\delta> 0$, we can choose 
$g\in C^\infty_c(W_1)$ s.t. 
$$||1- u_g||_{L^2(\Omega)}\leq \delta,$$ 
and for this chosen $g$
$$\frac{1}{\epsilon}|||a^{(1)}_l(x)- a^{(2)}_l(x)|^\frac{1}{l}u_{\epsilon, g}||_{L^2(\Omega)}\leq 
\frac{C}{\epsilon}||R^{(2)}_l(x, u_{\epsilon, g})- R^{(1)}_l(x, u_{\epsilon, g})||^{\frac{1}{l}}_{L^\infty(\Omega)}$$
$$\leq \frac{C'}{\epsilon}||u_{\epsilon, g}||^\frac{l+1}{l}_{C^s(\Omega)}\leq
C''\epsilon^\frac{1}{l}||g||^\frac{l+1}{l}_{C^2(\mathbb{R}^n)}$$
for small $\epsilon$ by Proposition 4.1 and 4.2.

Now let $\epsilon\to 0$, then we have
$$|||a^{(1)}_l(x)- a^{(2)}_l(x)|^\frac{1}{l}||_{L^2(\Omega)}\leq 
2\delta |||a^{(1)}_l(x)- a^{(2)}_l(x)|^\frac{1}{l}||_{L^\infty(\Omega)}.$$
Since $\delta$ is arbitrary, then $a^{(1)}_l= a^{(2)}_l$.
\end{proof}

\bibliographystyle{plain}
{\small\bibliography{Reference1}}
\end{document}